\documentclass [a4paper, 12pt, reqno]{amsart}
\usepackage [latin1]{inputenc}
\usepackage {a4}
\usepackage{amscd}
\usepackage{epsfig}
\usepackage{amssymb}
\usepackage{amsmath}
\usepackage{amsthm}
\usepackage[T1]{fontenc}
\usepackage{ae,aecompl}
\usepackage[arrow, matrix, curve]{xy}

\usepackage{color}

\usepackage{geometry}
\newcommand{\R} {\ensuremath{\mathbb{R}}}

\newcommand{\C} {\ensuremath{\mathbb{C}}}

\newcommand{\OO}{\mathcal{O}}
\renewcommand{\o}[1]{\overline{#1}}

\newcommand{\dq}{\overline{\partial}}

\DeclareMathOperator{\Sing}{Sing}

\DeclareMathOperator{\Jac}{Jac}

\newtheorem {satz} {Satz} [section]
\newtheorem {lem} [satz] {Lemma}
\newtheorem {cor} [satz] {Corollary}
\newtheorem {defn} [satz] {Definition}

\newtheorem {thm} [satz] {Theorem}

\DeclareMathOperator{\supp}{supp}

\renewcommand{\theta}{\vartheta}

\title[Parabolicity of the regular locus of complex varieties] 
{Parabolicity of the regular locus of complex varieties}

\author{J. Ruppenthal}

\address{Department of Mathematics, University of Wuppertal, Gau{\ss}str. 20, 42119 Wuppertal, Germany.}
\email{ruppenthal@uni-wuppertal.de}

\date{\today}



\begin{document}

~\\[-7mm]

\begin{abstract} 
The purpose of this note is to show that the regular locus of a complex variety is locally parabolic at the singular set.
This yields that the regular locus of a compact complex variety, e.g., of a projective variety, is parabolic.
We give also an application to the $L^2$-theory for the $\dq$-operator on singular spaces.
\end{abstract}

\maketitle

\section{Introduction}

There are many equivalent ways to define parabolicity of a Riemannian manifold.
Let us recall some of them:

\begin{defn}\label{defn:parabolic}
A Riemannian manifold $M$ is called {\bf parabolic} if the following equivalent conditions hold:

\smallskip
\begin{enumerate}
\item There exists a smooth exhaustion function $\psi\in C^\infty(M)$ with $\|d\psi\|_{L^2(M)}<\infty$.

\medskip
\item For each compact subset $K\subset M$ and each $\epsilon>0$, there exists a smooth cut-off function $\phi\in C^\infty_{cpt}(M)$
with $0\leq \phi\leq 1$, $\phi\equiv 1$ on a neighborhood of $K$ and $\|d\phi\|_{L^2(M)}<\epsilon$.

\medskip
\item Every subharmonic function on $M$ that is bounded from above is constant.

\medskip
\item There is no positive fundamental solution of the Laplacian on $M$, i.e.,
there is no positive Green function for the Laplacian on $M$.

\end{enumerate}
\end{defn}

\medskip
Condition (2) means that compact subsets of $M$ have zero capacity.
The equivalence of the conditions (1) -- (4) in Definition \ref{defn:parabolic} is standard knowledge
and can be derived directly e.g. from the characterization in \cite{GK} and \cite{G}.
We refer to \cite{Gri2}, \cite{Gri3} and \cite{GM} for the discussion of other sufficient and necessary conditions
(e.g. in terms of isoperimetric inequalities, Brownian motion, stochastic completeness, etc.)
and historical remarks.

Note that complete Riemannian manifolds are not necessarily parabolic.
By the Hopf-Rinow theorem, a Riemannian manifold is complete if and only if it carries
an exhaustion function with bounded gradient.
So, complete manifolds of finite volume are parabolic by condition (1) above.
More generally, Cheng and Yau \cite{CY} showed that complete Riemannian manifolds
are parabolic if the volume of geodesic balls grows at most like a quadratic polynomial.
A stronger sufficient condition is due to Grigor'yan \cite{Gri1}:
A complete manifold M is parabolic if
$\int_0^\infty \mbox{vol}(B_r(x_0))^{-1} r dr =\infty$,
where $B_r$ denotes the geodesic ball of radius $r$
around a fixed point $x_0\in M$.

For non-complete Riemannian manifolds, parabolicity can be characterized in terms of the "size"
of the "boundary" of the manifold.
Glasner \cite{G} showed that
non-compact Riemannian manifolds are parabolic if and only if they have a "small boundary" in the sense that Stokes' theorem holds:

\begin{thm}[{\bf Glasner \cite{G}}]
A non-compact Riemannian manifold $M$ of real dimension $n$ is parabolic if and only if Stokes' theorem is valid
for every square integrable $(n-1)$-form with integrable derivative, i.e., if
\begin{eqnarray*}
\int_M d\alpha &=& 0
\end{eqnarray*}
for every $\alpha \in L^2_{n-1}(M)$ with $d\alpha \in L^1_{n}(M)$.
\end{thm}

A similar characterization in terms of Green's formula in place of Stokes' theorem
is given by Grigor'yan and Masamune in \cite{GM}, Theorem 1.1.

\medskip
Let us now consider a Hermitian complex space $(X,h)$.
A Hermitian complex space $(X,h)$ is a paracompact reduced complex space $X$ with a metric $h$ on the regular locus
such that the following holds: If $x\in X$ is an arbitrary point, then there exists a neighborhood $U=U(x)$ and a
biholomorphic embedding of $U$ into a domain $G$ in $\C^N$ and an ordinary smooth Hermitian metric in $G$
whose restriction to $U$ is $h|_U$.
Examples are projective varieties with the restriction of the Fubini-Study metric or affine varieties
with the restriction of the Euclidean metric.
The Hermitian metric gives also a Riemannian structure on the regular locus (which is clearly not complete).
For a subset $\Omega\subset X$ and a differential form $\alpha$ on the regular part of $\Omega$,
we set
\begin{eqnarray*}
\|\alpha\|^2_{L^2(\Omega)} &=& \int_{\Omega\setminus \Sing X} |\alpha|_h^2\ dV_X,
\end{eqnarray*}
where $dV_X$ is the volume form on the regular part of $X$ with respect to the metric $h$.
We say that $\alpha\in L^{2,loc}(\Omega)$ if $\|\alpha\|_{L^2(K)} <\infty$ on each compact set $K\subset\Omega$.

\smallskip
As the singular set of a complex variety $X$ is of real codimension two,
thus "small" in a certain sense, it is reasonable to expect that the regular locus of $X$ is locally parabolic
at singular points. 
This idea is substantiated by Stokes' theorem for analytic varieties (see \cite{GH}, Chapter 0.2):

\begin{thm}\label{thm:GH}
Let $M$ be a complex manifold, $V\subset M$ an analytic subvariety of dimension $k$
and $\varphi$ a smooth differential form of degree $2k-1$ with compact support in $M$.
Then
\begin{eqnarray*}
\int_V d\varphi &=& 0.
\end{eqnarray*}
\end{thm}

Here, $\varphi$ and $d\varphi$ are bounded.
So, Theorem \ref{thm:GH} is not strong enough to imply parabolicity directly by Grasner's approach.
However, the principle behind the proof of Theorem \ref{thm:GH}
is to cut out the singular set and to show that the derivatives of a sequence of cut-off functions is uniformly bounded
in a sense that allows for Stokes' theorem to hold.
Refining such a cut-off procedure,
we were able to show that the regular locus of a Hermitian complex space is actually locally parabolic
(take $A=\Omega\cap\Sing X$ in the following):

\begin{thm}\label{thm:parabolic}
Let $X$ be a Hermitian complex space, $\Omega \subset X$ open and $A\subset \Omega$ a thin analytic subset that
contains the singular locus, i.e., $\Omega\cap \Sing X \subset A$.

\smallskip
\begin{enumerate}
\item Then there exists an exhaustion function $\phi\in C^\infty(\Omega\setminus A)$ of $\Omega\setminus A$
such that $d\phi \in L^{2,loc}(\Omega)$, i.e., $\|d\phi\|_{L^2(K)} < \infty$ for any compact set $K\subset\Omega$.

\medskip
\item Let $\Omega$ be relatively compact in $X$,
and let $K\subset \Omega\setminus A$ be a compact subset. Then there exists for each $\epsilon>0$
a smooth cut-off function $\psi\in C^\infty(\Omega)$ with $0\leq \psi\leq 1$ such that $\supp \psi \cap A=\emptyset$,
$\psi\equiv 1$ on a neighborhood of $K$ and $\|d\psi\|_{L^2(\Omega)} < \epsilon$.

\end{enumerate}
\end{thm}

\smallskip
We approach the question about parabolicity by means of condition (1) and (2) in Definition \ref{defn:parabolic},
because such exhaustion/cut-off procedures can be obtained by complex geometric techniques,
and because they play a crucial role in complex analysis on singular complex spaces
(see e.g. Theorem \ref{thm:bounded} below).

We will prove Theorem \ref{thm:parabolic} in Section \ref{sec:main1}.
Of particular interest is the case when $X$ is compact. Then $X$ can be given any Hermitian metric
as all such metrics on $X$ are equivalent, and we obtain:

\begin{cor}\label{cor:parabolic}
Let $X$ be a compact reduced complex space, e.g. a projective variety. Then the regular locus, $X\setminus \Sing X$, is parabolic.
\end{cor}

\smallskip
As an interesting consequence, all subharmonic functions on $X\setminus \Sing X$ that are bounded from above
must be constant.

Though Theorem \ref{thm:parabolic} and Corollary \ref{cor:parabolic} give nice examples of parabolic manifolds,
we are not aware of this statement from the literature. So, these notes may turn out useful as a reference.

\bigskip
Theorem \ref{thm:parabolic} has an interesting application to the $L^2$-theory for the $\dq$-operator on singular complex spaces.
Here, one considers $L^2$-forms on the regular locus of a Hermitian complex space $X$.
Due to the incompleteness of the metric on $X\setminus\Sing X$, there exist different closed $L^2$-extensions of the
$\dq$-operator on smooth forms with compact support on $X\setminus \Sing X$.
The two most important are the maximal and the minimal closed $L^2$-extension.

The maximal closed extension is the $\dq$-operator in the sense of distributions which we denote by $\dq_w$.
A differential form $f\in L^{2,loc}(\Omega)$ on an open set $\Omega\subset X$ is in the domain of $\dq_w$
if there exists a form $g\in L^{2,loc}(\Omega)$ such that $\dq f=g$ in the sense of distributions on $\Omega\setminus\Sing X$,
and we write $\dq_w f=g$ for that.

The minimal closed extension, denoted by $\dq_s$, is defined as follows. Let $f\in L^{2,loc}(\Omega)$ be in the domain of $\dq_w$.
Then we say that $f$ is in the domain of $\dq_s$ (and we set $\dq_s f:=\dq_w f$)
if there exists a sequence $\{f_j\}_j$ in the domain of $\dq_w$ such that
$$\supp f_j\cap \Sing X=\emptyset \ \ \ \forall j,$$
i.e., the $f_j$ have support away from the singular set, and
\begin{eqnarray}\label{eq:dqs1}
f_j &\rightarrow& f,\\
\dq_w f_j &\rightarrow& \dq_w f\label{eq:dqs2}
\end{eqnarray}
for $j\rightarrow \infty$ in $L^{2}(K)$ on each compact set $K\subset \Omega$.

So, the $\dq_s$-operator comes with a certain Dirichlet boundary condition (respectively growth condition) at the singular set of $X$.
It plays an important role in several studies of the $\dq$-operator on singular complex spaces (see \cite{PS}, \cite{OV}, \cite{R1}, \cite{R2}),
but it is very difficult to actually understand the domain of the $\dq_s$-operator.
Here now, we use the fact that the regular locus of $X$ is locally parabolic at singular points to deduce:

\begin{thm}\label{thm:bounded}
Bounded forms in the domain of the $\dq_w$-operator are also in the domain of the $\dq_s$-operator.
\end{thm}

This is interesting e.g. in the following context. If $X$ a $q$-complete Hermitian complex space of dimension $n$,
then $H^{0,n-q}_{\dq_s,cpt}(X)=0$ for $0 < q\leq n$, i.e. the $\dq_s$-equation with compact support is solvable for $(0,n-q)$-forms.
If $X$ has only rational singularities, then also $H^{0,q}_{\dq_s}(X)=0$ for $q>0$, i.e., the $\dq_s$-equation is solvable for $(0,q)$-forms
(see \cite{R2}, Theorem 1.3 and Theorem 1.6).
But Theorem \ref{thm:bounded} yields that bounded $\dq_w$-closed forms are also $\dq_s$-closed 
and so corresponding $\dq_s$-equations are solvable for bounded forms in the situations mentioned.

\medskip
After this note was accepted for publication,
N. Sibony pointed out
that Theorem \ref{thm:parabolic} is a special case of results that appeared in \cite{S}.
In fact, in \cite{S} it is shown that it is possible to cut out complete pluripolar sets
from positive closed $(p,p)$-currents with locally finite mass such that the $L^2$-norm of the derivatives of the cut-off functions is uniformly bounded.
The cut-off functions constructed in \cite{S}, Lemma 1.2, actually have the desired properties
(cf. the proof of \cite{S}, Theorem 1.1).

\medskip
{\bf Acknowledgements.} 
This research was supported by the Deutsche Forschungsgemeinschaft (DFG, German Research Foundation), 
grant RU 1474/2 within DFG's Emmy Noether Programme.
The author thanks Robert Berman for drawing his attention to the question of parabolicity of the regular locus of
complex varieties and for interesting discussions on the topic.
He also thanks an unknown referee for proposing to use a $\log\log$-exhaustion in Lemma \ref{lem:gradient}
which improved the proof of Theorem \ref{thm:parabolic} considerably,
and moreover Nessim Sibony for explaining how the results are contained in his more general theory.

\smallskip
\section{Resolution of singularities}\label{sec:resolution}

Let $X$ be a Hermitian complex space of pure dimension $n$ and $\pi: M\rightarrow X$ a resolution of singularities.
Let $h$ be the Hermitian metric on $X$ and $\gamma := \pi^* h$ its pull-back to $M$.
Then $\gamma$ is positive semidefinite (a pseudo-metric).
Let $M$ carry any (positive definite) metric $\sigma$. It follows that $\sigma \gtrsim \gamma$ on compact sets.

Let $dV_X$ be the volume form with respect to $h$ on $X$. Then $\pi^* dV_X = dV_\gamma$,
where $dV_\gamma$ is the volume form with respect to $\gamma$ on $M$.

\medskip
We consider a local coordinate patch in $M$ where we have coordinates $z_1, ..., z_n$.
Here, we can assume that $\sigma$ is just the Euclidean metric, i.e.,
$\langle \frac{\partial}{\partial z_j} , \frac{\partial}{\partial z_k} \rangle_\sigma = \delta_{jk}$.
Locally, $X$ has a holomorphic embedding $\iota: X \hookrightarrow \C^N$ into complex number space
such that $h$ is the pull-back of a regular Hermitian metric from $\C^N$ to $X$.
We can assume that $h$ is just the pull-back of the Euclidean metric, i.e., $h=\iota^* \langle\cdot,\cdot\rangle_{\C^N}$.
Choosing the coordinate patch on $M$, say $D=\{|z_j|\leq \frac{1}{2}: j=1, ...,n\}$, small enough,
we can assume that $\pi(D)$ is part of such an embedding and we can consider the resolution as a mapping
\begin{eqnarray*}
\pi=(\pi_1, ..., \pi_N): && D\subset \C^n \longrightarrow X \subset \C^N.
\end{eqnarray*}
So, $\gamma = \pi^* \langle\cdot,\cdot\rangle_{\C^N}$.
Let us describe that in other words. The Euclidean metric in $\C^N$ can expressed as $\sum_{j=1}^N dw_j \otimes d\o{w_j}$,
where $w_1, ..., w_N$ are the Euclidean coordinates.
Then
\begin{eqnarray*}
\gamma &=& \pi^* \sum_{j=1}^N dw_j \otimes d\o{w_j} = \sum_{j=1}^N d\pi_j \otimes d\o{\pi_j}.
\end{eqnarray*}
Let us express that in matrix notation. We denote by
\begin{eqnarray*}
\Jac \pi &=& \left( \frac{\partial\pi_j}{\partial z_k} \right)_{j,k}
\end{eqnarray*}
the Jacobian of $\pi$. Then
\begin{eqnarray*}
\langle v,w\rangle_\gamma &=& ^t\overline{v}\cdot\ ^t\overline{\Jac \pi}\cdot \Jac \pi\cdot w,
\end{eqnarray*}
i.e., $\gamma$ is represented (in the coordinates $z_1, ..., z_n$) by the Hermitian matrix 
\begin{eqnarray*}
H &=& ^t\o{\Jac \pi}\cdot \Jac\pi \geq 0.
\end{eqnarray*}
Let
\begin{eqnarray*}
H\ =\ \big( H_{jk}\big)_{j,k} \ \ \ ,\ \ \ H^{-1}\ =\ \big( H^{jk} \big)_{j,k}.
\end{eqnarray*}
We obtain
\begin{eqnarray}\label{eq:metric1}
\left| \frac{\partial}{\partial z_\mu}\right|^2 _\gamma = H_{\mu\mu} &,&
\left| dz_\mu\right|_\gamma^2 = H^{\mu\mu}
\end{eqnarray}
and
\begin{eqnarray}\label{eq:metric2}
\pi^* dV_X = dV_\gamma &=& \big( \det H\big) dV_{\C^n},\\
\big| dV_{\C^n} \big|_\gamma &=& \big(\det H\big)^{-1}.\label{eq:metric3}
\end{eqnarray}

\medskip
\noindent
From this we deduce easily a central lemma:

\begin{lem}\label{lem:metric1}
We have
\begin{eqnarray*}
\left|dz_\mu\right|_\gamma^2 &\lesssim& \left| dV_{\C^n}\right|_\gamma
\end{eqnarray*}
on compact sets.
\end{lem}

\begin{proof}
Let $H^\#$ be the adjungate matrix of $H$ such that $H^{-1}=(\det H)^{-1} H^\#$.
The entries of $H^\#$ are smooth functions on $M$.
By \eqref{eq:metric1} and \eqref{eq:metric3}, we obtain
\begin{eqnarray*}
\left|dz_\mu\right|_\gamma^2 = H^{\mu\mu} = (\det H)^{-1} H^\#_{\mu\mu} \lesssim (\det H)^{-1} = |dV_{\C^n}|_\gamma.
\end{eqnarray*}
\end{proof}

\smallskip
\section{Proof of Theorem \ref{thm:parabolic}}\label{sec:main1}

\subsection{$L^2$-Estimates for gradients of exhaustion functions}\label{sec:l2estimatescutoff}

\begin{lem}\label{lem:gradient}
Let $(X,h)$ be a Hermitian complex space and $f=(f_1, ..., f_m)$ a tuple of holomorphic functions on $X$
such that their common zero set is thin in $X$.

Let $r: \R_{\geq 0} \rightarrow [0,\frac{1}{2}]$ be a smooth function such that
$r(x)=x$ for $0\leq x\leq \frac{1}{4}$ and $r(x) = \frac{1}{2}$ for $x\geq \frac{1}{2}$. Let
\begin{eqnarray}\label{eq:loglog}
F &:=& \log | \log r\big(\sum_{k=1}^m |f_k|^2\big) |.
\end{eqnarray}

Then the $L^2$-norm of $\dq F$ is bounded on compact subsets of $X$, i.e., $\dq F \in L^{2,loc}(X)$.
The same statement holds for $\partial F$ and $d F$, respectively.
\end{lem}

\begin{proof}
Let $|f|^2:= \sum |f_k|^2$, and
let $\pi: M \rightarrow X$ be a resolution of singularities so that $\pi^* |f|^2$ becomes monomial in the following sense:
$M$ can be covered by patches $D=\{|z_j| \leq \frac{1}{2}: j=1, ..., n\}$ as in Section \ref{sec:resolution}
such that \eqref{eq:metric1} -- \eqref{eq:metric3} and Lemma \ref{lem:metric1} hold and
there is a non-vanishing smooth function $g$ on $D$ such that
\begin{eqnarray}\label{eq:pif}
\pi^* |f|^2 (z) &=& |z_1|^{2k_1} \cdots |z_d|^{2k_d}\cdot g(z)
\end{eqnarray}
with coefficients $k_1, ..., k_d\geq 1$ (this is possible because $|f|^2=\sum |f_j|^2$, where the $f_j$ are holomorphic).

It is enough to consider the $L^2$-norm of $\dq F$ on a compact set $\pi(D)$:
\begin{eqnarray*}
 \int_{\pi(D)\setminus\Sing X} \big| \dq F \big|^2_h dV_X \ =\
 \int_{D}  \big| \dq \pi^* F \big|^2_\gamma \pi^* dV_X
 \ =\ \int_{D} \big| \dq \pi^* F \big|^2_\gamma (\det A) dV_{\C^n}
\end{eqnarray*}
for such a patch $D$ (using \eqref{eq:metric2}).

By using $\pi^* F= \log|\log r( \pi^* |f|^2)|$, the representation \eqref{eq:pif} and Lemma \ref{lem:metric1}, we obtain:
\begin{eqnarray*}
\big| \dq \pi^* F \big|^2_\gamma &\lesssim& \frac{1}{\log^2 \big|r( \pi^*|f|^2)\big|} \cdot \frac{1}{\pi^* |f|^2} \cdot
\left| \sum_{j=1}^n \frac{\partial \pi^* |f|^2}{\partial \o{z_j}} d\o{z_j} \right|^2_\gamma\\
&\lesssim& \frac{1}{\log^2 \big|r( \pi^*|f|^2)\big|} \cdot \sum_{j=1}^n \frac{|d\o{z_j}|_\gamma^2}{|z_j|^2}\\
&\lesssim& \sum_{j=1}^n \frac{ |dV_{\C^n}|_\gamma}{|z_j|^2 \cdot \log^2 \big|r( \pi^*|f|^2)\big|}
\end{eqnarray*}
Taking into account that it is enough to estimate one summand (say for $j=1$),
and the fact that $\log^2 |r(\pi^* |f|^2)| \gtrsim \log^2 |z_1|^{2k_1}$ by use of \eqref{eq:pif},
the integral under consideration reduces to
\begin{eqnarray*}
  \int_{D} \frac{|dV_{\C^n}|_\gamma}{|z_1|^2 \log^2 |z_1|}  (\det A) dV_{\C^n} 
&=& \int_{D} \frac{dV_{\C^n}}{|z_1|^2 \log^2 |z_1|}  \ \lesssim\ 1,
\end{eqnarray*}
by use of \eqref{eq:metric3} in the first step. That proves the claim.

The same argument holds for $\partial$ or $d$ in place of $\dq$.
\end{proof}

\smallskip
\subsection{Proof of Theorem \ref{thm:parabolic}, statement (1)}

Let $X$ be a Hermitian complex space and $\Omega\subset X$ an open subset.
We first assume that there exists a tuple of holomorphic functions $f_1, ..., f_m \in\OO(\Omega)$
cutting out the thin analytic set $A$, i.e.,
\begin{eqnarray*}
A &=& \{f_1=...=f_m=0\}.
\end{eqnarray*}
Let
\begin{eqnarray}\label{defn:phiO}
\phi_\Omega &:=& \log | \log r\big( \sum_{j=1}^m |f_j|^2 \big) |
\end{eqnarray}
be the function defined as in Lemma \ref{lem:gradient}, \eqref{eq:loglog}.

Furthermore, let $\widehat\phi \in C^\infty(\Omega)$ be a smooth exhaustion function of $\Omega$.
Then
\begin{eqnarray*}
\phi &:=& \phi_\Omega + \widehat{\phi} \in C^\infty(\Omega\setminus A)
\end{eqnarray*}
is the desired exhaustion function by use of Lemma \ref{lem:gradient} and because $|d\widehat{\phi}|$ is locally bounded on $\Omega$.

\medskip
For the general case, let $\Omega\subset X$ be an arbitrary open set.
As $X$ is paracompact, we can cover $\o{\Omega}$ by a locally finite cover $\{\Omega_\nu\}_\nu$ of open sets $\Omega_\nu$
as above, and define exhaustion functions $\phi_{\Omega_\nu}$ as in \eqref{defn:phiO}.
Let $\psi_\nu$ be a partition of unity subordinate to $\{\Omega_\nu\}_\nu$.
Then
\begin{eqnarray}\label{defn:phiO2}
\phi_\Omega &:=& \sum_{\nu} \chi_\nu \phi_{\Omega_\nu} \ \in C^\infty(\Omega\setminus A)
\end{eqnarray}
satisfies $d\phi_\Omega \in L^{2,loc}(\Omega)$ and $\phi_\Omega(z) \rightarrow + \infty$ for $z\rightarrow z_0\in \Sing X$.
As above, $\phi_\Omega + \widehat{\phi}$ is the desired exhaustion function.

\medskip
\subsection{Proof of Theorem \ref{thm:parabolic}, statement (2)}

Let $\Omega$ be relatively compact and choose an open set $U\subset\subset X$
such that $\o{\Omega}\subset U$. Let $\phi_U$ be as in \eqref{defn:phiO2}. Thus,
\begin{eqnarray}\label{eq:est12}
\big\|d\phi_U\|_{L^2(\o{\Omega})} &<& \infty.
\end{eqnarray}
For $k\geq 1$, let $f_k: \R_{\geq 0} \rightarrow [0,1]$ be smooth functions such that $f_k(x)=1$ for $x\leq k$, $f_k(x)=0$ for $x\geq k+1$
and $|f'_k|\leq 2$. Now consider
\begin{eqnarray*}
\phi_k &:=& f_k\circ \phi_U \ \in C^\infty(U) \subset C^\infty(\Omega).
\end{eqnarray*}
Then $0\leq \phi_k \leq 1$, $\supp \phi_k\cap A=\emptyset$ and $\phi_k\equiv 1$ on a neighborhood of $K$
if $k$ is large enough. Moreover, we have
\begin{eqnarray*}
\big\| d\phi_k\|_{L^2(\o{\Omega})} &=& \| f_k'(\phi_U) d\phi_U \|_{L^2(\o{\Omega})} \leq 2 \|d\phi_U\|_{L^2(\o{\Omega})} < \infty.
\end{eqnarray*}
But the Lebesgue measure of $\o{\Omega} \cap \supp f_k'(\phi_U)$ vanishes for $k\rightarrow \infty$ and so
we obtain 
\begin{eqnarray*}
\big\| d\phi_k\|_{L^2(\o{\Omega})} &=& \| f_k'(\phi_U) d\phi_U \|_{L^2(\o{\Omega})} < \epsilon
\end{eqnarray*}
for $k$ large enough (see \cite{Alt}, A.1.16.2).
That proves the second part of Theorem \ref{thm:parabolic}.

\medskip
\section{Proof of Theorem \ref{thm:bounded}}\label{sec:main2}

Let $X$ be a Hermitian complex space,
$\Omega\subset X$ an open set and $\alpha\in L^{\infty}(\Omega)$ in the domain of the $\dq_w$-operator.
The question is local (see \cite{R1}, Section 6.1). Hence, 
we can assume that $\Omega$ is relatively compact in $X$ 
and that the singular set of $X$ is cut out in $\Omega$ by a tuple of holomorphic functions $f_1, ..., f_m \in \OO(\o{\Omega})$,
\begin{eqnarray*}
\Sing X \cap \o{\Omega} &=& \{f_1 = ...= f_m=0\}.
\end{eqnarray*}
By locality of the problem, it is moreover enough to show that $\alpha$ is in the domain of $\dq_s$ on small open
sets $V\subset\subset \Omega$, relatively compact in $\Omega$.

Let 
$$K_j:= \o{V} \setminus \{ |f| < 1/j\}$$
and choose by use of Theorem \ref{thm:parabolic}, (2), appropriate cut-off functions $\phi_j \in C^\infty(\Omega)$,
$0\leq \phi_j \leq 1$, such that $\supp \phi_j\cap \Sing X = \emptyset$,
$\phi_j \equiv 1$ in a neighborhood of $K_j$ and $\|d\phi_j\|_{L^2(\Omega)} < 1/j$.
Then
\begin{eqnarray*}
\alpha_j &:=& \phi_j \alpha
\end{eqnarray*}
is the required sequence (see \eqref{eq:dqs1} and \eqref{eq:dqs2}).

Note that
\begin{eqnarray*}
\dq_w \alpha_j &=& \phi_j  \dq_w \alpha + \dq \phi_j \wedge \alpha,
\end{eqnarray*}
where we can use the usual $\dq$ for the smooth cut-off functions $\phi_j$.

It is easy to see that $\alpha_j \rightarrow \alpha$ and $\phi_j \dq_w \alpha \rightarrow \dq_w \alpha$ in $L^2(K)$
for any compact set $K\subset V \subset\subset \Omega$ (by Lebesgue's theorem on dominated convergence).
Moreover, we have choosen the cut-off functions $\phi_j$ so that
\begin{eqnarray*}
\| \dq\phi_j \wedge \alpha\|_{L^2(K)} &\leq& \|\alpha\|_\infty \|d\phi_j\|_{L^2(K)} \leq \|\alpha\|_\infty / j \rightarrow 0
\end{eqnarray*}
for $j\rightarrow \infty$. That shows that $\alpha$ is actually in the domain of the $\dq_s$-operator.

\end{document}